\documentclass[12pt]{article}

\input xy

\xyoption{all}

\usepackage[utf8]{inputenc}
\usepackage[english]{babel}
\usepackage[T1]{fontenc}

\usepackage{sectsty,tocloft,hyperref,stmaryrd}

\usepackage[svgnames,x11names]{xcolor}

 \hypersetup{
   colorlinks, linkcolor={Blue},
   citecolor={red}, urlcolor={DarkOliveGreen4}
  }

\usepackage{enumitem}
\usepackage{mdwlist}
\usepackage{amssymb,amsfonts,amsbsy,amsthm,amsmath,graphicx,epsfig,bbm}
\usepackage{
times,mathrsfs}
\usepackage{pgf,tikz}
\usetikzlibrary{arrows}
\usepackage{textcomp}
\usepackage[many]{tcolorbox} 

\partfont{\normalsize}

\subsectionfont{\normalsize}

\subsubsectionfont{\mdseries\itshape\normalsize}

\newenvironment{eq}{\begin{equation}}{\end{equation}} 

\theoremstyle{theorem}
\newtheorem{thm}{Theorem}[section]
\newtheorem{cor}{Corollary}[thm]

\newtheorem{lem}[thm]{Lemma}
\newtheorem{prop}[thm]{Proposition}

\newtheorem{defi}[thm]{Definition}

\theoremstyle{definition}
\newtheorem{rmq}[thm]{Remark}

\newtheorem{exe}[thm]{Example}

\theoremstyle{remark}

\newenvironment{rmq*}
 {\pushQED{\qed}\rmq}
 {\popQED\endrmq}

\newenvironment{exe*}
 {\pushQED{\qed}\exe}
 {\popQED\endexe}

\renewcommand{\bf}[1]{\boldsymbol{#1}}

\newcommand{\bbC}{\mathbb{C}}

\newcommand{\bbE}{\mathbb{E}}

\newcommand{\bbN}{\mathbb{N}}
\newcommand{\bbP}{\mathbb{P}}

\newcommand{\bbR}{\mathbb{R}}
\newcommand{\bbT}{\mathbb{T}}
\newcommand{\bbU}{\mathbb{U}}
\newcommand{\bbZ}{\mathbb{Z}}

\newcommand{\bfX}{\mathbf{X}}
\newcommand{\bfY}{\mathbf{Y}}
\newcommand{\bfZ}{\mathbf{Z}}

\renewcommand{\d}{\mathrm{d}}

\newcommand{\A}{\mathscr{A}}
\newcommand{\B}{\mathscr{B}}

\newcommand{\F}{\mathscr{F}}

\newcommand{\calH}{\mathcal{H}}

\renewcommand{\P}{\mathscr{P}}
\newcommand{\calP}{\mathcal{P}}

\newcommand{\calT}{\mathcal{T}}

\renewcommand{\O}{\Omega}

\newcommand{\e}{\epsilon}

\renewcommand{\k}{\kappa}
\renewcommand{\l}{\lambda}
\renewcommand{\o}{\omega}
\newcommand{\s}{\sigma}

\newcommand{\Id}{\mathrm{Id}}

\newcommand{\Exp}{\mathrm{Exp}}

\renewcommand{\t}[1]{\widetilde{#1}}

\newcommand{\bbm}[1]{\mathbbm{#1}}
\renewcommand{\frak}[1]{\mathfrak{#1}}
\renewcommand{\to}{\longrightarrow}
\newcommand{\arr}{\rightarrow}
\renewcommand{\phi}{\varphi}

\newcommand*{\quotient}[2]
{\ensuremath{
    #1/\!\raisebox{-.65ex}{\ensuremath{#2}}}}

\def\res#1#2{\mathchoice
              {\setbox1\hbox{${\displaystyle #1}_{\scriptstyle #2}$}
              \restrictionaux{#1}{#2}}
              {\setbox1\hbox{${\textstyle #1}_{\scriptstyle #2}$}
              \restrictionaux{#1}{#2}}
              {\setbox1\hbox{${\scriptstyle #1}_{\scriptscriptstyle #2}$}
              \restrictionaux{#1}{#2}}
              {\setbox1\hbox{${\scriptscriptstyle #1}_{\scriptscriptstyle #2}$}
              \restrictionaux{#1}{#2}}}
\def\restrictionaux#1#2{{#1\,\smash{\vrule height 1\ht1 depth 1.1\dp1}}_{\,#2}} 

\def\commutatif{\ar@{}[rd]|{\circlearrowleft}}

\definecolor{main}{HTML}{5989cf}
\newtcolorbox{boxB}{
    fontupper = \bf, 
    boxrule = 1.5pt,
    colframe = main,
    rounded corners,
    arc = 5pt   
}

\begin{document}

\title{Confined Poisson extensions}

\author{Séverin Benzoni, Emmanuel Roy, Thierry de la Rue
\\}
\date{}

\maketitle

\begin{abstract}
This paper follows on from \cite{benzoniConfined}, where we introduced the notion of \emph{confined extensions}, and our purpose is to widen the context in which such extensions appear. We do so in the setup of Poisson suspensions: we take a $\s$-finite measure-preserving dynamical system $(X, \mu, T)$ and a compact extension $(X \times G, \mu \otimes m_G, T_\phi)$, then we consider the corresponding Poisson extension $((X \times G)^*, (\mu \otimes m_G)^*, (T_\phi)_*) \overset{}{\to} (X^*, \mu^*, T_*)$. Our results give two different conditions under which that extension is confined. Finally, to show that those conditions are not void, we give an example of a system $(X, \mu, T)$ and a cocycle $\phi$ such that the compact extension $(X \times G, \mu \otimes m_G, T_\phi)$ has an infinite ergodic index.
\end{abstract}

\tableofcontents

\section{Introduction}

\subsection{Motivations}

This paper investigates the concept of extensions of measure-preserving dynamical systems, specifically, extensions given by a factor map $\pi: (Z, \rho, R) \arr (X, \mu, T)$. We mean that $\bfZ := (Z, \rho, R)$ and $\bfX := (X, \mu, T)$ are invertible measure preserving dynamical systems on standard Borel sets such that $\bfX$ is a factor of $\bfZ$ via $\pi$, and conversely, we also view $\bfZ$ as an extension of $\bfX$. 

This paper is a continuation of the work done in \cite{benzoniConfined}. There, we introduced the notion of confined extensions: they are extensions $(Z, \rho, R) \overset{\pi}{\to} (X, \mu, T)$ such that the only self-joining $\l$ of $(Z, \rho, R)$ in which the law of $\pi \times \pi$ is the product measure $\mu \otimes \mu$, is the product joining $\l = \rho \otimes \rho$ (see Definition \ref{def:confined}).

This notion was first of interest to us in the study of dynamical filtrations, which are filtrations defined on some dynamical system $(X, \mu, T)$ of the form $\F := (\F_n)_{n \leq 0}$ such that each $\F_n$ is $T$-invariant (see \cite{benzoniConfined} for more details). But we also noticed other interesting results on confined extensions. For example, we listed properties $\calP$ that are lifted through confined extensions, i.e. if $(X, \mu, T)$ satisfies $\calP$ and $(Z, \rho, R) \overset{\pi}{\to} (X, \mu, T)$ is confined, then $(Z, \rho, R)$ satisfies $\calP$ (see \cite[Section 3.3]{benzoniConfined}). 

Since we noticed that confined extensions had many interesting properties, we look for examples in which that behavior arises. In \cite{benzoniConfined}, we considered extensions well known in the literature, namely, compact extensions and $T, T^{-1}$-transformations. In both cases, we gave necessary and sufficient conditions for those extensions to be confined. 

In this paper, we give confinement results for a new kind of extension, in the setting of Poisson suspensions. Take $(X, \mu, T)$ a measure preserving dynamical system where $\mu$ is a $\s$-finite measure, but assume that $\mu(X) = \infty$. Consider the probability space $(X^*, \mu^*)$ where $X^*$ is the set of locally finite counting measures of the form $\sum_{i \geq 1} \delta_{x_i}$, with $(x_i)_{i \geq 1} \in X^\bbN$, and $\mu^*$ the law of the Poisson process of intensity $\mu$. One can then define $T_*$ on $(X^*, \mu^*)$ by applying $T$ to each point of the point process. The resulting dynamical system $(X^*, \mu^*, T_*)$ is called the \emph{Poisson suspension over $(X, \mu, T)$}. A factor map $\pi: (Z, \rho, R) \arr (X, \mu, T)$ between infinite measure systems will then yield a factor map between the Poisson suspensions: $\pi_*: (Z^*, \rho^*, R_*) \arr (X^*, \mu^*, T_*)$. The resulting extension is what we call a \emph{Poisson extension}.

We will consider the case where $\bfZ := (X \times G, \mu \otimes m_G, T_\phi)$ is the compact extension given by a cocycle $\phi: X \arr G$, with $G$ a compact group. We recall that the map $T_\phi$ is defined as 
$$T_\phi(x, g) := (Tx, g \cdot \phi(x)).$$
 Our results concern the following Poisson extension:
\begin{eq} \label{eq:poisson_ext}
((X \times G)^*, (\mu \otimes m_G)^*, (T_\phi)_*) \overset{\pi_*}{\to} (X^*, \mu^*, T_*),
\end{eq}
with $\pi: (x, g) \mapsto x$.

In Section \ref{sect:trivial_cocycle}, we consider the case where $\phi(x)$ acts as the identity map, for every $x \in X$. Using a splitting result from \cite{ErgPoissonSplittings}, we prove that in this case, if $(X, \mu, T)$ is of infinite ergodic index, the extension \eqref{eq:poisson_ext} is confined (see Theorem \ref{thm:T_Id_confined}).

In Section \ref{sect:compact_poisson_ext}, we deal with a more general cocycle $\phi$. There, our argument will rely on the assumption that the compact extension $(X \times G, \mu \otimes m_G, T_\phi)$ is of infinite ergodic index. In that case, we make use of Lemma \ref{lem:furstenberg}, which is a well know result from Furstenberg. Through some intricate manipulations, we manage to reduce our problem to a relative unique ergodicity problem for products of the extension $\bfZ \overset{\pi}{\to} \bfX$, so that we can use Furstenberg's lemma (i.e. Lemma \ref{lem:furstenberg}) to prove that \eqref{eq:poisson_ext} is confined (see Theorem \ref{thm:conf_T_phi}).

Since the argument developed in Section \ref{sect:compact_poisson_ext} requires a compact extension $(X \times G, \mu \otimes m_G, T_\phi)$ of infinite ergodic index, in Section \ref{sect:ext_infinte_erg_index}, we give an example of such an extension, showing that Theorem \ref{thm:conf_T_phi} is not void.

\subsection{Basic notions and notation in ergodic theory}
\label{sect:notation}

A \emph{dynamical system} is a quadruple $\bfX := (X, \A, \mu, T)$ such that $(X, \A)$ is a standard Borel space, $\mu$ is a Borel measure which is $\s$-finite, i.e. there exist measurable sets $(X_n)_{n \geq 1}$ such that $\mu(X_n) < \infty$ and $X = \bigcup_{n \geq 1} X_n$, and $T$ is an invertible measure-preserving transformation. Throughout the paper, we will often not specify the $\s$-algebra $\A$, and will write our dynamical systems as a triple of the form $(X, \mu, T)$. 

If we have two systems $\bfX := (X , \mu, T)$ and $\bfZ := (Z, \rho, R)$, a \emph{factor map} is a measurable map $\pi: Z \to X$ such that $\pi_*\rho = \mu$ and $\pi \circ R = T \circ \pi$, $\rho$-almost surely. If such a map exists, we say that $\bfX$ is a \emph{factor} of $\bfZ$. Conversely, we also say that $\bfZ$ is an extension of $\bfX$. Moreover, if there exist invariant sets $X_0 \subset X$ and $Z_0 \subset Z$ of full measure such that $\pi: Z_0 \to X_0$ is a bijection, then $\pi$ is an \emph{isomorphism} and we write $\bfZ \cong \bfX$.

The system $(X, \mu, T)$ is \emph{ergodic} if $T^{-1} A = A$ implies that $\mu(A)=0$ or $\mu(A^c)=0$. It is \emph{conservative} if there is no non-trivial set $A$ such that the $\{T^nA\}_{n \in \bbZ}$ are disjoint. Let $(X, \mu, T)$ be a dynamical system with $\mu(X) = \infty$. If $(X, \mu, T)^{\otimes k}$ is conservative and ergodic, so are all the smaller exponents $(X, \mu, T)^{\otimes j}$, $j \leq k$. The \emph{ergodic index} of $(X, \mu, T)$ is the largest integer $k$ such that $(X, \mu, T)^{\otimes k}$ is conservative and ergodic. If $(X, \mu, T)^{\otimes k}$ is ergodic for every integer $k$, the ergodic index is infinite.

Let $(X, \mu, T)$ be a conservative system. For $A \subset  X$ measurable, we denote the restriction of $\mu$ to $A$ by $\res{\mu}{A} := \mu(\cdot \cap A)$. Since the system is conservative, the return time $N_A(x) := \inf\{k \geq 1 \, | \, T^k \in A\}$ is almost surely finite, allowing us to define the induced transformation 
$$\begin{array}{cccc}
\res{T}{A}: & A & \to & A\\
& x & \longmapsto & T^{N_A(x)}x
\end{array}.$$
Moreover, we have the following:
\begin{lem} \label{lem:conservative_factor}
Let $\bfZ := (Z, \rho, R)$ and $\bfX := (X, \mu, T)$ be two $\s$-finite systems and $\pi: \bfZ \to \bfX$ be a factor map. Then $\bfZ$ is conservative if and only if $\bfX$ is conservative. 
\end{lem}

\subsection{Joinings and confined extensions}

Let $\bfX := (X, \mu, T)$ and $\bfY := (Y, \nu, S)$ be two $\s$-finite measure preserving dynamical systems. A \emph{joining} of $\bfX$ and $\bfY$ is a $(T \times S)$-invariant measure $\l$ on $X \times Y$ whose marginals are $\mu$ and $\nu$ (therefore the marginals have to be $\s$-finite). It yields the dynamical system
$$\bfX \times_\l \bfY := (X \times Y, \l, T \times S).$$
On this system, the coordinate projections are factor maps that project onto $\bfX$ and $\bfY$ respectively. If it is not necessary to specify the measure, we will simply write $\bfX \times \bfY$. For the product joining, we will use the notation $\bfX \otimes \bfY := \bfX \times_{\mu \otimes \nu} \bfY$. For the $n$-fold product self-joining, we will write $\bfX^{\otimes n}$. 

When $\bfX$ and $\bfY$ are probability measure preserving systems, there is at least one joining, the product joining $\bfX \otimes \bfY$. However, if we have infinite measures, the product measure is not a joining because its marginals are not $\s$-finite. In fact, there exist pairs of systems for which there does not exist any joining. For example, Lemma \ref{lem:conservative_factor} implies that there cannot exist a joining of a conservative and a non-conservative system. 

We now give the definition of \emph{confined extensions}, which concerns only probability measure preserving dynamical systems.
\begin{defi} \label{def:confined}
Let $\bfX := (X, \mu, T)$ and $\bfY := (Y, \nu, S)$ be probability measure preserving dynamical systems, and $\pi: \bfX \to \bfY$ be a factor map. The extension $\bfX \overset{\pi}{\to} \bfY$ is said to be \emph{confined} if it satisfies one of the following equivalent properties:
\begin{enumerate}[label = (\roman*)]
\item every $2$-fold self-joining of $\bfX$ in which the two copies of $\pi$ are independent random variables is the product joining: i.e. the only measure $\l \in \P(X \times X)$ that is $T \times T$-invariant, with $\l(\cdot \times X) = \l(X \times \cdot)= \mu$ and $(\pi \times \pi)_*\l = \nu \otimes \nu$, is $\l = \mu \otimes \mu$;
\item for every system $\bfZ$, every joining of $\bfX$ and $\bfZ$ in which the copy of $\pi$ and the projection on $\bfZ$ are independent random variables is the product joining;
\item for every $n \in \bbN^* \cup \{+ \infty\}$, every $n$-fold self-joining of $\bfX$ in which the $n$ copies of $\pi$ are mutually independent random variables is the $n$-fold product joining.
\end{enumerate}
\end{defi}
It was shown in \cite[Proposition 3.3]{benzoniConfined} that the definitions (i), (ii), and (iii) are equivalent. In this paper, we mainly use the definition (i). As we mentioned, this concerns only the case for probability measures. An adaptation to the infinite measure case would be more intricate, mainly because if we assume that a measure $\l$ on $X \times X$ projects onto $\nu \otimes \nu$ on $Y \times Y$ and that $\nu$ is an infinite measure, then $\l$ cannot be a joining of $\mu$. That is because, in that case, both projections of $\l$ on $X$ would not be $\s$-finite. 

\subsection{Compact extensions and relative unique ergodicity}

\begin{defi}
Let $(Z, \rho, R) \overset{\pi}{\to} (X, \mu, T)$ be an extension. It is \emph{relatively uniquely ergodic} if the only $R$-invariant measure $\l$ on $Z$ such that $\pi_*\l = \mu$ is $\l = \rho$.
\end{defi}

Let $\bfX := (X, \mu, T)$ be a measure preserving dynamical system, $G$ a compact group and $\phi: X \to G$ a cocycle. Let $m_G$ denote the Haar probability measure on $G$. The compact extension of $\bfX$ given by $\phi$ is the system $\bfZ$ on $(X \times G, \mu \otimes m_G)$ given by the skew product 
$$\begin{array}{cccc}
T_\phi: & X \times G & \to & X \times G\\
& (x, g) & \longmapsto & (Tx, g \cdot \phi(x))
\end{array}.$$
This is the most well known family of extensions. The only result we will need, is the following, due to Furstenberg:
\begin{lem}[Furstenberg \cite{Furstenberg_book}] \label{lem:furstenberg}
Let $\bfX := (X, \mu, T)$ be an ergodic measure preserving dynamical system where $\mu$ is a finite or $\s$-finite measure. Assume that the compact extension $\bfZ = (X \times G, \mu \otimes m_G, T_\phi)$ is ergodic. Let $\l$ be a $\s$-finite $T_\phi$-invariant measure on $X \times G$ such that $\l(\cdot \times G) = \mu$. Then
$$\l = \mu \otimes m_G.$$ 
\end{lem}
This lemma is usually stated with $\mu$ a probability measure, but the infinite measure case is proven in the exact same way. 

Furstenberg's lemma can be summarized by saying that an ergodic compact extension is relatively uniquely ergodic.

\subsection{Poisson suspensions, splittings and extensions }
\label{sect:intro_poisson}

Let $\bfX := (X, \mu, T)$ be a $\s$-finite measure preserving dynamical system. For convenience, we will assume that $X = \bbR^+$ and that $\mu$ is a locally finite measure, i.e. for any bounded set $B \subset \bbR^+$, we have $\mu(B) < \infty$. We define the set of counting measures on $X$ by
$$X^* := \left\{\text{locally finite measures of the form } \sum_{i \geq 1} \delta_{x_i}\right\}.$$
For technical properties of $X^*$, we follow the setup of \cite[Section 1.1]{sushi}. In particular, the measurable sets on $X^*$ are the restrictions of the Borel sets of $\t X$, the space of locally finite measures on $X$ (see \cite[§ A2.6]{Daley_Vere-Jones}).

A point process is a probability measure on $X^*$. The Poisson point process of intensity $\mu$, which we denote $\mu^*$, is the point process characterized by the fact that, for $A_1, ..., A_n \subset X$ measurable disjoint subsets such that $0 < \mu(A_i) < \infty$, the random variables $\o(A_1), ..., \o(A_n)$, for $\o \in X^*$, are independent Poisson random variables of respective parameter $\mu(A_i)$, for $i \in \llbracket 1, n \rrbracket$. 

On the probability space $(X^*, \mu^*)$, we define the transformation 
$$T_*: \sum_{i \geq 1} \delta_{x_i} \mapsto \sum_{i \geq 1} \delta_{Tx_i}.$$
This map might not be defined for every $\o \in X^*$ (since $T_*(\o)$ is not necessarily locally finite), but it is a well-defined bijective transformation on $\bigcap_{n \in \bbZ} T_*^{-n}X^*$, which is a set of full measure for $\mu^*$. Once again, this follows the setup of \cite[Section 1.1]{sushi}. The resulting dynamical system $\bfX^* := (X^*, \mu^*, T_*)$ is called the \emph{Poisson suspension over} $(X, \mu, T)$. 

 It is well known that the Poisson suspension $\bfX^*$ is ergodic if and only if there is no $T$-invariant measurable subset $A \subset X$ such that $0 < \mu(A) < \infty$ (see \cite{marchat}). Moreover, this implies that if $\bfX^*$ is ergodic, it is automatically weakly mixing. Also, note that it is not necessary that $\bfX$ is ergodic for $\bfX^*$ to be ergodic.

We use the notion of Poisson splittings from \cite{ErgPoissonSplittings}, but with different choices in the notation. A \emph{splitting of order $n$} of the Poisson suspension $(X^*, \mu^*, T_*)$ is a family $\{\nu_i\}_{1 \leq i \leq n}$ of $T_*$-invariant probability measures on $X^*$ and $\l$ a $T_*^{\times n}$-invariant joining of $\{\nu_i\}_{1 \leq i \leq n}$ such that $\Sigma^{(n)}_*\l = \mu^*$, where 
$$\begin{array}{cccc}
\Sigma^{(n)}: & X^* \times \cdots \times X^* & \to & X^*\\
& (\o_1, ..., \o_n) & \longmapsto & \o_1 + \cdots + \o_n
\end{array}.$$
The splitting is said to be \emph{ergodic} if $\l$ is an ergodic joining. The splitting is a \emph{Poisson splitting} if there exist $\{\mu_i\}_{1 \leq i \leq n}$, $\s$-finite measures on $X$ such that, for $i \in \llbracket 1, n \rrbracket$, $\nu_i = \mu_i^*$, and $\l$ is the product measure $\mu_1^* \otimes \cdots \otimes \mu_n^*$. With that notation, the result \cite[Theorem 2.6]{ErgPoissonSplittings} becomes 
\begin{thm} \label{thm:poisson_splitting}
Let $\bfX := (X, \mu, T)$ be a $\s$-finite measure preserving dynamical system of infinite ergodic index. Then any ergodic splitting of the Poisson suspension $(X^*, \mu^*, T_*)$ is a Poisson splitting.
\end{thm}

Consider two $\s$-finite systems $\bfZ := (Z, \rho, R)$ and $\bfX := (X, \mu, T)$ and a factor map $\pi: \bfZ \to \bfX$, which means that we have an extension $\bfZ \overset{\pi}{\to} \bfX$ of $\s$-finite systems. We can then define the map 
$$\pi_*: \sum_{i \geq 1} \delta_{x_i} \mapsto \sum_{i \geq 1} \delta_{\pi (x_i)}.$$
One can check that this yields a factor map from $\bfZ^*$ to $\bfX^*$, therefore we have defined an extension $\bfZ^* \overset{\pi_*}{\to} \bfX^*$ between Poisson suspensions. Such an extension is what we call a \emph{Poisson suspension}.

\section{A Poisson extension over a trivial cocycle}
\label{sect:trivial_cocycle}

In this section, we study Poisson extensions over extensions of the form 
$$\begin{array}{cccc}
T \times \Id: & X \times K & \to & X \times K\\
& (x, \k) & \longmapsto & (Tx, \k)
\end{array},$$
on $(X \times K, \mu \otimes \rho)$, where $K$ is a standard Borel space and $\rho$ is a probability measure on $K$. We start in Section \ref{sect:rue_and_confined} by showing that if $T$ has infinite ergodic index, the associated Poisson extension is confined. Then in Section \ref{sect:perm_group}, we see that marked point processes enable us to write Poisson extensions through a Rokhlin cocycle, and we give an application in probability theory by giving an alternative proof of the De Finetti theorem (see Corollary \ref{cor:definetti}). Finally, in Section \ref{sect:non-confined}, we give an example of a non-confined Poisson extension.

\subsection{Confinement as a consequence of Poisson splittings}
\label{sect:rue_and_confined}

We derive the content of this section as a consequence of Theorem \ref{thm:poisson_splitting}. In \cite{ErgPoissonSplittings}, the authors proved Theorem \ref{thm:poisson_splitting} and gave an application of that splitting result (specifically, \cite[Theorem 3.1]{ErgPoissonSplittings}). Here, we note that it can be rephrased as a relative unique ergodicity result for the Poisson extension. In our notation, it becomes:
\begin{thm} \label{cor:poisson_splitting}
Let $\bfX := (X, \mu, T)$ be a $\s$-finite measure preserving dynamical system of infinite ergodic index, and $K$ a standard Borel space. Let $\l$ be an invariant marked point process over $\mu^*$, i.e. a $(T \times \Id)_*$-invariant probability measure on $(X \times K)^*$ such that $(\pi_*)_*\l = \mu^*$. If $(\l, (T \times \Id)_*)$ is ergodic, then there exists a probability measure $\rho$ on $K$ such that $\l = (\mu \otimes \rho)^*$. 
\end{thm}

We deduce that the Poisson extension is confined:
\begin{thm} \label{thm:T_Id_confined}
Let $\bfX := (X, \mu, T)$ be a $\s$-finite measure preserving dynamical system of infinite ergodic index, and $(K, \rho)$ a standard probability space. Then the Poisson extension
$$((X \times K)^*, (\mu \otimes \rho)^*, (T \times \Id)_*) \to (X^*, \mu^*, T_*),$$
is confined. 
\end{thm}
\begin{proof}
Set $\bfZ := ((X \times K)^*, (\mu \otimes \rho)^*, (T \times \Id)_*)$ and $\pi: (x, \k) \mapsto x$. Let $\l$ be a $2$-fold self joining of $\bfZ$ such that $(\pi_* \times \pi_*)_*\l = \mu^* \otimes \mu^*$. Since $\bfZ^*$ and $\bfX^*$ are ergodic, and even weakly mixing (see Section \ref{sect:intro_poisson} or \cite{marchat}), up to taking an ergodic component, one can assume that $\l$ is ergodic. Note that $\l$ is a probability measure on 
$$(X \times K)^* \times (X \times K)^*.$$
Both marginals of $\l$ on $(X \times K)^*$ yield a Poisson point process of intensity $\mu \otimes \rho$. We view the realization of both of those processes simultaneously on $X \times K$ and we tag the points coming from the first coordinate with a $1$, and the points coming from the second coordinate with a $2$. To do that formally, we define the map
$$\begin{array}{cccc}
\O:& (X \times K)^* \times (X \times K)^* & \to & (X \times K \times \{1, 2\})^*\\
& (\o_1, \o_2) & \longmapsto & \o_1 \otimes \delta_{\{1\}} + \o_2 \otimes \delta_{\{2\}}
\end{array},$$
so that 
\begin{eq} \label{eq:map_poisson-sum}
\O(\o_1, \o_2)( \cdot \times \{i\}) = \o_i.
\end{eq}
Consider $\eta := \O_*\l$. Since $\O \circ (T \times \mathrm{Id}_K)_* \times (T \times \mathrm{Id}_K)_* = (T \times \mathrm{Id}_{K \times \{1, 2\}})_* \circ \O$, we know that $\eta$ is $(T \times \Id_{K \times \{1, 2\}})_*$-invariant. Moreover, because $\l$ is assumed to be ergodic, the system $((X \times K \times \{1, 2\})^*, \eta, (T \times \mathrm{Id}_{K \times \{1, 2\}})_*)$ is ergodic. Finally, we need to look at the projection on $X^*$ via $\tilde\pi: (x, \k, i) \mapsto x$. To that end, we note that $(\tilde\pi_*)_*\eta = (\tilde\pi_* \circ \O)_*\l$, and 
\begin{align*}
\tilde\pi_* \circ \O(\o_1, \o_2) &= \tilde\pi_*(\o_1 \otimes \delta_{\{1\}} + \o_2 \otimes \delta_{\{2\}})\\
&= \tilde\pi_*(\o_1 \otimes \delta_{\{1\}}) + \tilde\pi_*(\o_2 \otimes \delta_{\{2\}}) = \pi_*\o_1 + \pi_*\o_2.
\end{align*}
However, $(\pi_* \times \pi_*)_*\l = \mu^* \otimes \mu^*$, which means that, in the above notation, $\pi_*\o_1$ and $\pi_*\o_2$ are independent Poisson processes of intensity $\mu$. It is known that the sum of such two independent Poisson processes is a Poisson process of intensity $2 \mu$. Therefore, $(\tilde\pi_*)_*\eta = (\tilde\pi_* \circ \O)_*\l = (2\mu)^*$. Theorem \ref{cor:poisson_splitting} tells us that there exists $\chi \in \P(K \times \{1, 2\})$ such that $\eta = (2\mu \otimes \chi)^*$. 

Now we show that $\chi = \rho \otimes (\frac{1}{2}(\delta_{\{1\}} + \delta_{\{2\}}))$. Let $A \subset X$ such that $0 < \mu(A) < \infty$, $B \subset K$ and $i \in \{1, 2\}$ 
\begin{align*}
e^{-2\mu(A) \chi(B \times \{i\})} &= \eta(\{\tilde \o \, ; \, \tilde \o(A \times B \times \{i\}) = 0\})\\
&= \l( \{(\o_1, \o_2) \, ; \, \o_i(A \times B) = 0\}) \; \text{ because of } \eqref{eq:map_poisson-sum}\\
&= (\mu \otimes \rho)^*(\{ \o \, ; \, \o(A \times B) = 0\}) = e^{-\mu(A) \rho(B)}.
\end{align*}
So, $\chi(B \times \{i\}) = \frac{1}{2} \rho(B)$. Therefore $\chi = \rho \otimes (\frac{1}{2}(\delta_{\{1\}} + \delta_{\{2\}}))$, so $\eta = (2\mu \otimes \rho \otimes (\frac{1}{2}(\delta_{\{1\}} + \delta_{\{2\}})))^*$. Finally, we get 
\begin{align*}
\l = \O^{-1}_*\eta &= \O^{-1}_*(2\mu \otimes \rho \otimes (\frac{1}{2}(\delta_{\{1\}} + \delta_{\{2\}})))^*\\
&= (\mu \otimes \rho)^* \times (\mu \otimes \rho)^*.
\end{align*}
\end{proof}

\subsection{Marked Point processes}
\label{sect:marked_point_process}

Let $(X, \mu)$ be a standard Borel space equipped with a $\s$-finite measure such that $\mu(X) = \infty$. Without loss of generality, we can assume that $X = \bbR_+$, thus enabling us to use the natural order on $\bbR^+$, but any other order could be used here. We may also assume that $\mu$ is the Lebesgue measure on $\bbR_+$ (by doing so, we ignore the case where $\mu$ has atoms, but for the rest of our work, that is not a problem). Up to a set of $\mu^*$-measure $0$ we can assume that the elements $\o$ of $(\bbR_+)^*$ are locally finite measures with no multiplicity, i.e. such that $\forall x \in \bbR_+, \, \o(\{x\}) \leq 1$. This allows us to define a sequence $(t_n)_{n \in \bbN}$ of measurable maps from $(\bbR_+)^*$ to $\bbR_+$ such that 
$$\o = \sum_{n \geq 1} \delta_{t_n(\o)},$$
and 
$$0 \leq t_1(\o) < t_2(\o) < \cdots.$$
Each $t_n(\o)$ gives us the position of the $n$-th atom of the counting measure $\o$.

Now consider a Polish space $K$. We will call a \emph{marked point process over $\mu^*$} a probability measure $\l$ on $(\bbR_+ \times K)^*$ such that $(\pi_*)_*\l = \mu^*$, where $\pi: (x, \kappa) \mapsto x$. We already manipulated marked point processes in the previous section, we are simply giving them a name now. We can describe marked point processes as follows: define the map 
$$
\begin{array}{cccc}
f: & (\bbR_+)^* \times K^\bbN & \to & (\bbR_+ \times K)^*\\
& (\o, (\kappa_n)_{n \geq 1}) & \longmapsto & \sum_{n \geq 1} \delta_{(t_n(\o), \kappa_n)}\\
\end{array}.
$$
Since $f$ is injective, we know that $f((\bbR_+)^* \times K^\bbN)$ is a Borel set and $f^{-1}$ is measurable, and we can write it as 
\begin{eq} \label{eq:iso_marked_point_process}
\Phi := f^{-1}: \tilde \o \mapsto (\pi_*(\tilde\o), (\kappa_n(\tilde\o))_{n \geq 1}),
\end{eq}
where $(\kappa_n(\tilde\o))_{n \geq 1}$ is called the sequence of the \emph{marks} of $\tilde \o$. For a marked point process $\l$, $\l\big(f((\bbR_+)^* \times K^\bbN) \big) =1$, therefore, up to a set of measure $0$, $f$ is a bijection. Moreover, we have the following result, from \cite[Lemma 6.4.VI]{Daley_Vere-Jones}:
\begin{prop} \label{prop:marked_point_process}
Let $\rho$ be a probability measure on $K$. The Poisson process $(\mu \otimes \rho)^*$ is a marked point process over $\mu^*$ and $f_*(\mu^* \otimes \rho^{\otimes \bbN}) = (\mu \otimes \rho)^*$. 
\end{prop}
In other words, if $\tilde \o \in (\bbR_+ \times K)^*$ is distributed according to the Poisson process of intensity $\mu \otimes \rho$, the sequence of marks $(\kappa_n(\tilde\o))_{n \geq 1}$ is i.i.d. of law $\rho$.

\subsection{A $\mathfrak{S}(\bbN)$-valued cocycle and its action on $K^\bbN$}
\label{sect:perm_group}

In Section \ref{sect:marked_point_process}, we saw that, assuming that $X = \bbR_+$, a point process on $(X \times K)^*$ can be represented on $X^* \times K^\bbN$, via the map introduced in \eqref{eq:iso_marked_point_process}:
$$\begin{array}{cccc}
\Phi: & (X \times K)^* & \to & X^* \times K^\bbN\\
& \tilde \o & \longmapsto & (\pi_*(\tilde\o), (\k_n(\tilde \o))_{n \geq 1})
\end{array}.$$
where $\k_n(\tilde \o)$ is the mark associated to the $n$-th point of $\tilde\o$, once the points of $\tilde\o$ are ordered according to their projection on $X$.
Now we mean to determine the dynamic on $X^* \times K^\bbN$ that would correspond to $(T \times \Id)^*$ on $(X \times K)^*$. To do this, we will need a tool to track how the order of the points of $\o \in X^*$ changes when $T_*$ is applied. 

The group $\mathfrak{S}(\bbN)$ is the group of the permutations of $\bbN$, i.e. the bijections from $\bbN$ onto itself. Equipped with the metric
$$d(\s, \tau) := \sum_{n \in \bbN} \frac{1}{2^n} \bbm1_{\s(n) \neq \tau(n)},$$
it is a Polish group, i.e. $(\frak{S}(\bbN), d)$ is a complete separable metric space and the map $(\s, \tau) \mapsto \s \circ \tau^{-1}$ is continuous. This group acts on $K^\bbN$ via the measurable action
\begin{eq} \label{eq:action_perm}
(\s, (\k_n)_{n \geq 1}) \mapsto (\k_{\s^{-1}(n)})_{n \geq 1}.
\end{eq}

We recall that, given $\o \in X^*$, we denote by $(t_n(\o))_{n \geq 1}$ the ordered sequence of the points of $\o$. To describe the action of $T_*$ on $(t_n(\o))_{n \geq 1}$, we define $\Psi(\omega)$ as the unique element of $\frak{S}(\bbN)$ such that for every $n \in \bbN$, $\Psi(\omega)(n)$ is the rank of the atom $T(t_n(\omega))$ in the counting measure $T_*\omega$, i.e. $T(t_n(\omega))$ is the $\Psi(\omega)(n)$-th atom of $T_*\omega$ when counting from left to right. This defines a cocycle
$$\Psi: X^* \to \mathfrak{S}(\bbN),$$
so that 
$$T(t_n(\o))= t_{\Psi(\o)(n)}( T_*\o).$$
We consider the skew-product define by the cocycle $\Psi$:
$$\begin{array}{cccc}
(T_*)_\Psi: & X^* \times K^\bbN & \to & X^* \times K^\bbN\\
& (\o, (\k_n)_{n \geq 1}) & \longmapsto & (T_*\o, (\k_{\Psi(\o)^{-1}(n)})_{n \geq 1})
\end{array}.$$
Then we check that 
\begin{eq} \label{eq:equivariance}
\begin{split}
\Phi \circ (T \times \Id)_* (\t\o) &= (\pi_*(T \times \Id)_* \t\o, (\k_n((T \times \Id)_* \t\o))_{n \geq 1})\\
&= (T_*\pi_*\t\o, (\k_{\Psi(\o)^{-1}(n)}(\t\o))_{n \geq 1})\\
&= (T_*)_\Psi(\pi_*\t\o, (\k_n(\t\o))_{n \geq 1})\\
&= (T_*)_\Psi \circ \Phi (\t\o).
\end{split}
\end{eq}
Combined with Proposition \ref{prop:marked_point_process}, this tells us that $\Phi$ is an isomorphism between the extensions
$$((X \times K)^*, (\mu \otimes \rho)^*, (T \times \Id)_*) \to (X^*, \mu^*, T_*),$$
and
$$(X^* \times K^\bbN, \mu^* \otimes \rho^{\otimes \bbN}, (T_*)_\Psi) \to (X^*, \mu^*, T_*).$$
Through this isomorphism, Theorem \ref{cor:poisson_splitting} becomes 
\begin{thm} \label{cor_poisson_splitting_bis}
Let $\bfX := (X, \mu, T)$ be a $\s$-finite measure preserving dynamical system of infinite ergodic index, and $K$ a standard Borel space. Let $\l$ be a $(T_*)_\Psi$-invariant probability measure on $X^* \times K^\bbN$ such that $\l( \cdot \times K^\bbN) = \mu^*$. If $(\l, (T_*)_\Psi)$ is ergodic, then there exists a probability measure $\rho$ on $K$ such that $\l = \mu^* \otimes \rho^{\otimes \bbN}$. 
\end{thm}
As an unexpected corollary, we get the following result, which is the De Finetti theorem written in the language of ergodic theory:
\begin{cor}[De Finetti, Hewitt-Savage] \label{cor:definetti}
Let $\rho_\infty$ be a $\frak{S}(\bbN)$-invariant (under the action defined by \eqref{eq:action_perm}) probability measure on $K^{\bbN}$ and $\rho$ its marginal on the first coordinate. The action $(K^\bbN, \rho_\infty, \frak{S}(\bbN))$ is ergodic if and only if $\rho_\infty = \rho^{\otimes \bbN}$.
\end{cor}
\begin{proof}
If $\rho_\infty = \rho^{\otimes \bbN}$, it is ergodic under the action of $\frak{S}(\bbN)$ for the same reason that it is ergodic under the shift. Following this remark, the proof is quite straightforward.

Assume that $(K^\bbN, \rho_\infty, \frak{S}(\bbN))$ is ergodic. Let $(X, \mu, T)$ be a dynamical system of infinite ergodic index. Since $\rho_\infty$ is $\frak{S}$-invariant, it follows that $\mu^* \otimes \rho_\infty$ is $(T_*)_\Psi$-invariant. We write the ergodic decomposition of $\mu^* \otimes \rho_\infty$ by 
$$\mu^* \otimes \rho_\infty = \int_{\P(X^* \times K^{\bbN})} \l d\mathbb{P}(\l).$$
Then Theorem \ref{cor_poisson_splitting_bis} tells us that the ergodic decomposition of $\mu^* \otimes \rho_\infty$ is of the form 
$$\mu^* \otimes \rho_\infty = \int_{\P(K)} \mu^* \otimes \gamma^{\otimes \bbN} d\t\bbP(\gamma) = \mu^* \otimes \int_{\P(K)} \gamma^{\otimes \bbN} d \t\bbP(\gamma),$$
where $\t\bbP := \Theta_*\bbP$ and $\Theta$ sends a measure on $X^* \times K^{\bbN}$ to its projection on the first coordinate on $K$. So $\rho_\infty = \int_{\P(K)} \gamma^{\otimes \bbN} d \t\bbP(\gamma)$. However, each measure $\gamma^{\otimes \bbN}$ is $\frak{S}(\bbN)$-invariant, and $\rho_\infty$ is ergodic under $\frak{S}(\bbN)$. Therefore, there exists $\gamma \in \P(K)$ such that $\rho_\infty = \gamma^{\otimes \bbN}$. 
\end{proof}

\subsection{A non-confined Poisson extension}
\label{sect:non-confined}

We give here an example of a non-trivial non-confined Poisson extension, to show that the infinite ergodic index assumption in Theorem \ref{thm:T_Id_confined} cannot be removed. Take $X := \bbR$, $\mu$ the Lebesgue measure on $\bbR$, and 
$$T: x \mapsto x+1.$$
The system $(X, \mu, T)$ is not ergodic and not conservative and its ergodic index is $0$, but the Poisson suspension $(X^*, \mu^*, T_*)$ is ergodic, since it is Bernoulli. We get the following:
\begin{prop}
Let $(K, \rho)$ be a standard probability space. The extension $((X \times K)^*, (\mu \otimes \rho)^*, (T \times \Id)_*) \arr (X^*, \mu^*, T_*)$ is not confined. 
\end{prop}
\begin{proof}
We will make use of the setup presented in the previous section for the study of marked point processes. We make some slight adjustments since now $X = \bbR$ (instead of $\bbR_+$): define a sequence $(t_n(\o))_{n \in \bbZ}$ such that 
$$\o = \sum_{n \in \bbZ} \delta_{t_n(\o)},$$
and
$$\cdots t_{-1}(\o) < t_0(\o) < 0 \leq t_1(\o) < t_2(\o) \cdots.$$
We then also define a cocycle 
$$\t\Psi: X^* \to \frak{S}(\bbZ),$$
such that 
$$T(t_n(\o)) = t_{\t\Psi(\o)(n)}(T_*\o).$$
In other words, $\t\Psi(\o)(n)$ is the rank of the atom $T(t_n(\o))$ in $T_*\o$. In our present case, the map $\t\Psi$ can be described explicitly: denote the shift $S: k \mapsto k+1$ and then one can check that
$$\t\Psi: \o \mapsto S^{\o([-1, 0[)}
.$$
Indeed, the shift in the numbering of the atoms of $\o$ is only affected by the atoms in $[-1, 0[$ as they go from being smaller than $0$ to being greater than $0$ once we apply $T$. 

As in the previous section, we get an isomorphism in between the extensions
$$((X \times K)^*, (\mu \otimes \rho)^*, (T \times \Id)_*) \to (X^*, \mu^*, T_*),$$
and
$$(X^* \times K^\bbZ, \mu^* \otimes \rho^{\otimes \bbZ}, (T_*)_{\t\Psi}) \to (X^*, \mu^*, T_*).$$
We prove our proposition by showing that the second extension is not confined. We need to build a non-product self-joining of $\mu^* \otimes \rho^{\otimes \bbZ}$ whose projection via $\t\pi \times \t\pi$ is $\mu^* \otimes \mu^*$. It will be more convenient to describe this joining as a measure on $X^* \times X^* \times K^\bbZ \times K^\bbZ$. We start with the marginal on $X^* \times X^*$ which has to be $\l( \cdot \times \cdot \times K^\bbZ \times K^\bbZ) = \mu^* \otimes \mu^*$, and for $(\o_1, \o_2) \in X^* \times X^*$, the conditional law $\l_{(\o_1, \o_2)}$ is as follows. The sequence of marks of $\o_1$, $(\k_n(\o_1))_{n \in \bbZ}$ is chosen with probability $\rho^{\otimes \bbZ}$. For the choice of $\k_n(\o_2)$, we distinguish two situations:
\begin{itemize}
\item If $\o_1([t_n(\o_2), t_{n+1}(\o_2)[) \geq 1$, we set 
$$\ell := \min\{k \in \bbZ \, | \, t_k(\o_1) \in [t_n(\o_2), t_{n+1}(\o_2)[\},$$
and then we choose $\k_n(\o_2) := \k_\ell(\o_1)$. 
\item If $\o_1([t_n(\o_2), t_{n+1}(\o_2)[) = 0$, we choose $\k_n(\o_2)$ with law $\rho$, independently from all the other marks.
\end{itemize}
The construction of $\l$ is concluded by taking
$$\l := \int \delta_{\o_1} \otimes \delta_{\o_2} \otimes \l_{(\o_1, \o_2)} \, d(\mu^* \otimes \mu^*)(\o_1, \o_2).$$
Note that our choices for $(\k_n(\o_1))_{n \in \bbZ}$ and $(\k_n(\o_2))_{n \in \bbZ}$ depend only on the relative positions of the points $\{t_n(\o_1)\}_{n \in \bbZ}$ and $\{t_n(\o_2)\}_{n \in \bbZ}$. Since those relative positions are preserved under application of $(T_* \times T_*)$, the measure $\l$ is $(T_*)_{\t\Psi} \times (T_*)_{\t\Psi}$-invariant (up to a permutation of coordinates).

From our construction, it is clear that $\l$ is not a product measure and that 
$$\l(\cdot \times X^* \times \cdot \times K^\bbZ) = \mu^* \otimes \rho^{\otimes \bbZ}.$$
We are left with checking that $\l(X^* \times \cdot \times K^\bbZ \times \cdot) = \mu^* \otimes \rho^{\otimes \bbZ}$. Consider that $\o_1, \o_2$ and $(\k_n(\o_2))_{n < n_0}$ are known and compute the law of $\k_{n_0}(\o_2)$: if $\o_1([t_n(\o_2), t_{n+1}(\o_2)[) = 0$, it follows from our construction that the law of $\k_{n_0}(\o_2)$ is $\rho$. If $\o_1([t_n(\o_2), t_{n+1}(\o_2)[) \geq 1$, we have $\k_n(\o_2) = \k_\ell(\o_1)$ (see above for the definition of $\ell$). One can check that $(\k_n(\o_2))_{n < n_0}$ only informs us on (some of) the values of $(\k_n(\o_1))_{n < \ell}$ and $\k_\ell(\o_1)$ is independent of $\o_1, \o_2$ and $(\k_n(\o_1))_{n < \ell}$. So, even with $\o_1, \o_2$ and $(\k_n(\o_2))_{n < n_0}$ fixed, the law of $\k_\ell(\o_1)$ is $\rho$, so the law of $\k_n(\o_2)$ is also $\rho$. 

To sum up, up to a permutation of coordinates, $\l$ is a $(T_*)_{\t\Psi} \times (T_*)_{\t\Psi}$-invariant measure on $X^* \times K^\bbZ \times X^*\times K^\bbZ$, such that 
$$\l(\cdot \times \cdot \times X^* \times K^\bbZ) = \mu^* \otimes \rho^{\otimes \bbZ} \, \text{ and } \, \l(X^* \times K^\bbZ \times \cdot \times \cdot) = \mu^* \otimes \rho^{\otimes \bbZ}.$$
So it a self-joining of $\mu^* \otimes \rho^{\otimes \bbZ}$. Moreover, it projects onto $\mu^* \otimes \mu^*$ without being equal to the product measure $\mu^* \otimes \rho^{\otimes \bbZ} \otimes \mu^* \otimes \rho^{\otimes \bbZ}$. This precisely means that the extension 
$$(X^* \times K^\bbZ, \mu^* \otimes \rho^{\otimes \bbZ}, (T_*)_{\t\Psi}) \to (X^*, \mu^*, T_*),$$
is not confined.
\end{proof}

\section{A Poisson suspension over a compact extension}
\label{sect:compact_poisson_ext}

In this section, we are interested in the Poisson extension a over compact extension, i.e. over the system $\bfZ$ given by
$$\begin{array}{cccc}
T_\phi: & X \times G & \to & X \times G\\
& (x, g) & \longmapsto & (Tx, g \cdot \phi(x))
\end{array},$$
for some compact group $G$ and measurable cocycle $\phi: X \arr G$. Our goal will be to show that 
\begin{thm} \label{thm:conf_T_phi}
Let $\bfX := (X, \mu, T)$ be a dynamical system of infinite ergodic index. If the compact extension $(X \times G, \mu \otimes m_G, T_\phi)$ is also of infinite ergodic index, then the Poisson extension 
$$((X \times G)^*, (\mu \otimes m_G)^*, (T_\phi)_*) \overset{\pi_*}{\to} (X^*, \mu^*, T_*)$$
is confined.
\end{thm}

We start with the Sections \ref{sect:keane} and \ref{sect:distinguish_points}, where we introduce some useful notions and results from the literature. We then prove the main technical step in the proof of our theorem in Section \ref{sect:thm_conf_T_phi}. We conclude in Section \ref{sect:final_proof_conf_Tphi}
.

\subsection{Ergodicity of Cartesian products in spectral theory}
\label{sect:keane}

We present briefly some results on the ergodicity of Cartesian products of $\s$-finite measure preserving dynamical systems. 

We start by introducing some classic notions in spectral theory. Let $(X, \mu, T)$ be a $\s$-finite measure preserving dynamical system. Consider the space $L^2(X, \mu)$ and the Koopman operator
$$\begin{array}{cccc}
U_T: & L^2(X, \mu) & \to & L^2(X, \mu)\\
& f & \longmapsto & f \circ T
\end{array}.$$

If $\mu$ is a finite measure, denote $L^2_0(X, \mu) \subset L^2(X, \mu)$ the subspace orthogonal to the space of constant maps. For $f \in L^2(X, \mu)$, the spectral measure of $f$, $\s_f$, is defined as the measure on $\bbU$ such that, for every $n \in \bbZ$
$$\widehat{\s_f} (n) = \int_X f \, \overline{U_T^n f} d\mu.$$
There exists a finite measure $\s_X^0$ on $\bbU$, unique up to equivalence, such that 
\begin{itemize}
\item for every $f \in L^2_0(X, \mu)$, $\s_f \ll \s_X^0$,
\item and for every finite measure $\s$ such that $\s \ll \s_X^0$, there exists $f \in L^2_0(X, \mu)$ such that $\s = \s_f$.
\end{itemize}
 It is the \emph{restricted maximal spectral type} of $(X, \mu, T)$. 

If $\mu$ is a $\s$-finite measure, we define a $L^\infty$-eigenvalue as $\l \in \bbC$ such that there exists $f \in L^\infty(X, \mu)$ non-zero a.e. such that 
$$f \circ T = \l f.$$
Such a map $f$ is called a $L^\infty$-eigenfunction. Denote $e(T)$ the set of $L^\infty$-eigenvalues of $(X, \mu, T)$, which is a sub-group of $\bbU$, provided $T$ is conservative and ergodic (see \cite[Section 2.6]{aaronson}). The notion of $L^\infty$-eigenvalues is mainly useful in the infinite measure case. Indeed, if $\mu(X) < \infty$, we have $L^\infty(X, \mu) \subset L^2(X, \mu)$, so $L^\infty$-eigenfunctions are simply eigenvectors of the Koopman operator $U_T$.

We will use the following ergodicity criterion, due to Keane (see \cite[Section 2.7]{aaronson}):
\begin{thm} \label{thm:keane}
Let $\bfX := (X, \mu, T)$ be a conservative and ergodic dynamical system and $\bfY := (Y, \nu, S)$ be an ergodic probability measure preserving dynamical system. The Cartesian product $\bfX \otimes \bfY$ is ergodic if and only if $\s_Y^0(e(T)) = 0$. 
\end{thm}
In \cite[Theorem 1.2]{Meyerovitch}, Meyerovitch uses this criterion to prove the following
\begin{thm}[Meyerovitch] \label{thm:meyerovitch_poisson}
Let $\bfX := (X, \mu, T)$ be a conservative and ergodic dynamical system with $\mu(X) = \infty$. The product system 
$$(X, \mu, T) \otimes (X^*, \mu^*, T_*)$$ 
is ergodic if and only if $\bfX$ is ergodic.
\end{thm}
We will need the following generalization, which comes as a corollary. We thank the referees for suggesting this proof.
\begin{cor} \label{cor:keane_poisson}
Let $\bfX := (X, \mu, T)$ be a conservative and ergodic dynamical system with $\mu(X) = \infty$ and $k \geq 1$. The product system 
$$(X, \mu, T)^{\otimes k} \otimes (X^*, \mu^*, T_*)$$ 
is ergodic if and only if $\bfX^{\otimes k}$ is ergodic.
\end{cor}
\begin{proof}
Let $k \geq 1$. Assume that $\bfX^{\otimes k}$ is ergodic. We show below that 
\begin{eq} \label{eq:spectre_meyerovitch}
e(T^{\times k}) = e(T).
\end{eq}
Using Theorem \ref{thm:keane}, this implies that $(X, \mu, T)^{\otimes k} \otimes (X^*, \mu^*, T_*)$ is ergodic if and only if $(X, \mu, T) \otimes (X^*, \mu^*, T_*)$ is ergodic. Then Meyerovitch's theorem ends our proof. (Note that we need $\bfX^{\otimes k}$ to be ergodic to prove \eqref{eq:spectre_meyerovitch}).

Let us now prove \eqref{eq:spectre_meyerovitch} by proving that, if $\bfX := (X, \mu, T)$ and $\bfY := (Y, \nu, S)$ are dynamical systems such that $\bfX \otimes \bfY$ is ergodic, we have $e(T \times S) = e(T) e(S)$. Let $\l \in e(T \times S)$ and $F \in L^\infty(X \times Y, \mu \otimes \nu)$ an associated eigenfunction. We have 
$$F \circ (T \times \Id) \circ (T \times S) = F \circ (T \times S) \circ (T \times \Id) = \l F \circ (T \times \Id).$$
So $F \circ (T \times \Id)$ is also an eigenfunction associated to $\l$ and, since $\bfX \otimes \bfY$ is ergodic, there exists $c_{T \times \Id} \in \bbC$ such that $F \circ (T \times \Id) = c_{T \times \Id} F$. By Fubini's theorem, $c_{T \times \Id} \in e(T)$. Similarly, there is also $c_{\Id \times S} \in e(S)$ such that $F \circ (\Id \times S) = c_{\Id \times S} F$. Finally, note that
$$\l F = F \circ (T \times S) = F \circ (T \times \Id)\circ (\Id \times S) = c_{T \times\Id} c_{\Id \times S} F.$$
Therefore, $\l = c_{T \times\Id} c_{\Id \times S} \in e(T) e(S)$. Hence $e(T \times S) \subset e(T) e(S)$, and the converse inclusion is clear. 
\end{proof}

\subsection{Distinguishing points in a Poisson process}
\label{sect:distinguish_points}

In this section, we use the sequence $(t_n(\o))_{n \geq 1}$ introduced in Section \ref{sect:marked_point_process}. Because of that, we need $X$ to be $\bbR^+$ and the measure $\mu$ to be the Lebesgue measure. Some additional aspects of the structure of $(\bbR^+, \mu)$ will also be useful. The purpose of this section is to study the map 
$$\begin{array}{cccc}
\tilde\Phi_k: & X^k \times X^* & \to & X^*\\
& (x_1, ..., x_k, \o) & \longmapsto & \delta_{x_1} + \cdots +\delta_{x_k} + \o\\
\end{array}.$$
We view the points of $X^k \times X^*$ as a Poisson process for which the first $k$ points are distinguished, so that we can track each of them individually. To avoid any multiplicity on the right-hand term, we will study this map on a smaller set $X^{(k)} \subset X^k \times X^*$, defined as
$$X^{(k)} := \{(x_1, ..., x_k, \o) \in X^k \times X^* \, | \, x_1 < \cdots < x_k < t_1(\o)\}.$$
From now on, $\Phi_k$ denotes the restriction of $\tilde \Phi_k$ to $X^{(k)}$. We start by computing the measure of $X^{(k)}$, using the fact $t_1$ follows an exponential law of parameter $1$:
\begin{align*}
\mu^{\otimes k} \times \mu^*(X^{(k)}) & = \int_{X^*} \int_{(\bbR_+)^k} \mathbbm{1}_{x_1 < \cdots < x_k < t_1(\o)} d\mu(x_1) \cdots d\mu(x_k) d\mu^*(\o)\\
&= \int_{\bbR_+}  \int_{(\bbR_+)^k} \mathbbm{1}_{x_1 < \cdots < x_k < t} d\mu(x_1) \cdots d\mu(x_k) e^{-t} dt\\
&= \int_{\bbR_+} \frac{t^k}{k!} e^{-t} dt =1,
\end{align*}
the last equality being obtained through $k$ successive integrations by parts. We complete this with the following result
\begin{lem} \label{prop:distinguish_points}
Let $k \geq 1$. The map $\Phi_k$ sends $\res{(\mu^{\otimes k} \otimes \mu^*)}{X^{(k)}}$ onto $\mu^*$. Therefore
$$\Phi_k: (X^{(k)}, \res{(\mu^{\otimes k} \otimes \mu^*)}{X^{(k)}}) \to (X^*, \mu^*),$$
is an isomorphism of probability spaces.
\end{lem}

\begin{proof}
It is clear that $\Phi_k$ is a bijection whose inverse is 
$$\o \mapsto (t_1(\o), ..., t_k(\o), \o - (\delta_{t_1(\o)} + \cdots + \delta_{t_k(\o)})).$$

We then need to prove that $\Phi_k$ is measure-preserving. We prove that result by induction on $k$. The case $k=1$ can be found in \cite[Proposition 6.1]{Meyerovitch}, but we give a proof for completeness. Denote $\Exp$ the law of an exponential variable of parameter $1$. 

To prove that $(\Phi_1)_* \res{(\mu \otimes \mu^*)}{X^{(1)}} = \mu^*$, we need to show that, if $(x, \o)$ is chosen under $\res{(\mu \otimes \mu^*)}{X^{(1)}}$, the sequence 
$$\left(x, t_1(\o) -x, \left(t_{i+1}(\o) - t_i(\o)\right)_{i \geq 1}\right)$$
 is i.i.d. of law $\Exp$. First, we know that $\left(t_{i+1} - t_i\right)_{i \geq 1}$ is i.i.d. of law $\Exp$. It is also clear that $(x, t_1 -x)$ is independent from $\left(t_{i+1} - t_i\right)_{i \geq 1}$. Therefore, we now only have to compute the law of $(x, t_1 -x)$ under $\res{(\mu \otimes \mu^*)}{X^{(1)}}$. Let $A, \, B \subset \bbR_+$ be measurable sets. We have:
\begin{align*}
\mu \otimes \mu^*(x < t_1, x \in A,& t_1 - x \in B) = \int_A \mu^*(x <t_1, t_1 - x \in B) d\mu(x)\\
&= \int_A \underset{=e^{-x}}{\underbrace{\mu^*(x<t_1)}} \underset{=\mu^*(t_1 \in B)}{\underbrace{\mu^*(t_1-x \in B \, | \, t_1 > x)}} d\mu(x)\\
&= \int_A e^{-x} d\mu(x) \mu^*(t_1 \in B) = \Exp(A) \cdot \Exp(B),
\end{align*}
where we use the fact that the law of $t_1$ is $\Exp$, and the loss of memory property of $\Exp$. 

Let $k \geq 1$, and assume that the result is true for $k$. We start by noting that $\Phi_{k+1} = \Phi_1 \circ (\Id \times \Phi_k)$ and use the induction hypothesis to prove that 
\begin{eq} \label{eq:push_X^k}
(\Id \times \Phi_k)_*\res{(\mu^{\otimes k+1} \otimes \mu^*)}{X^{(k+1)}} = \res{(\mu \otimes \mu^*)}{X^{(1)}}.
\end{eq}
Indeed, for a measurable map, $F: X^{(1)} \arr \bbR$, we have 
\begin{align*}
\int_{X^{(k+1)}}& F(x_1, \delta_{x_2} + \cdots + \delta_{k+1} + \o) \, d\mu^{\otimes k+1}(x_1, ..., x_{k+1}) d\mu^*(\o)\\
&= \int_{\bbR_+}\int_{X^{(k)}} \mathbbm{1}_{x_1 < x_2} F(x_1, \sum_{i=2}^{k+1} \delta_{x_i} + \o) \, d\mu^{\otimes k}(x_2, ..., x_{k+1}) d\mu^*(\o) d\mu(x_1)\\
&=\int_{\bbR_+}\int_{X^*} \mathbbm{1}_{x_1 < t_1(\o)} F(x_1, \o) d\mu^*(\o) d\mu(x_1)\\
&= \int_{X^{(1)}} F d\mu \otimes \mu^*,
\end{align*}
by the induction hypothesis and the fact that $x_2 = t_1(\delta_{x_2} + \cdots \delta_{k+1} + \o)$. Therefore \eqref{eq:push_X^k} is proven. We then combine it with the result for $k = 1$ to conclude that 
$$(\Phi_{k+1})_*\res{(\mu^{\otimes k+1} \otimes \mu^*)}{X^{(k+1)}} = (\Phi_1)_* \res{(\mu \otimes \mu^*)}{X^{(1)}} = \mu^*.$$
\end{proof}

We now want to study how $\Phi_k$ matches the dynamics on $X^{(k)}$ and $X^*$. We recall that we defined $\Psi: (\bbR_+)^* \arr \mathfrak{S}(\bbN)$ so that $T(t_n(\o))  = t_{\Psi(\o)(n)}(T_*\o)$. We then iterate it to define 
$$\Psi_p(\o) := \Psi(T_*^{p-1}\o) \circ \cdots \circ \Psi(\o).$$
This iteration means that $\Psi_p(\o)(n)$ is the rank of the atom $T^p(t_n(\o))$ in the counting measure $T^p_*\o$. Now consider 
$$N^{(k)}(\o) := \inf\{ p \geq 1 \, | \, \Psi_p(\o)(1) = 1, ..., \Psi_p(\o)(k) = k\}.$$ 
This is the first time in which the first $k$ points of $\o$ are back to being the first $k$ points of $T_*^p\o$ and in their original order. If the random time $N^{(k)}$ is almost surely finite, we can define the automorphism $T_*^{N^{(k)}}$ on $(X^*, \mu^*)$ by 
$$\left(T_*^{N^{(k)}}\right)(\o) := T_*^{N^{(k)}(\o)}(\o).$$
We conclude this section with the following result:
\begin{prop} \label{thm:distinguish_points}
Let $\bfX := (X, \mu, T)$ be a $\s$-finite measure preserving dynamical system. Assume that $T$ has infinite ergodic index. Then, for any $k \geq 1$, $N^{(k)}$ is $\mu^*$-almost surely well defined and $\Phi_k$ is an isomorphism between the systems 
$$(X^{(k)}, \res{(\mu^{\otimes k} \otimes \mu^*)}{X^{(k)}}, \res{(T^{\times k} \times T_*)}{X^{(k)}})$$
and 
$$(X^*, \mu^*, T_*^{N^{(k)}}).$$
We recall that $\res{(T^{\times k} \times T_*)}{X^{(k)}}$ is the induced transformation on $X^{(k)}$.
\end{prop}
\begin{proof}
Let $k \geq 1$. Since $\bfX = (X, \mu, T)$ is of infinite ergodic index, the system $\bfX^{\otimes k} = (X^k, \mu^{\otimes k}, T^{\times k})$ is conservative and ergodic. Since $\mu^*$ is a probability measure, one can check that the system $(X^k \times X^*, \mu^{\otimes k} \otimes \mu^*, T^{\times k} \times T_*)$ is also conservative. Therefore, the induced system 
$$(X^{(k)}, \res{(\mu^{\otimes k} \otimes \mu^*)}{X^{(k)}}, \res{(T^{\times k} \times T_*)}{X^{(k)}})$$
is well defined. Moreover, if $M^{(k)}$ is the first return time in $X^{(k)}$, then $M^{(k)}$ is $\res{(\mu^{\otimes k} \otimes \mu^*)}{X^{(k)}}$-almost surely finite. However, since we have 
\begin{eq} \label{eq:tphi_invaraince}
\tilde\Phi_k \circ (T^{\times k} \times T_*) = T_* \circ \tilde\Phi_k,
\end{eq}
one can check that on $X^{(k)}$, we have 
\begin{eq} \label{eq:return_times}
N^{(k)} = M^{(k)} \circ \Phi_k.
\end{eq}
So, because Lemma \ref{prop:distinguish_points} shows that $(\Phi_k)_*\res{(\mu^{\otimes k} \otimes \mu^*)}{X^{(k)}} = \mu^*$, we deduce that $N^{(k)}$ is $\mu^*$-almost surely finite. Finally, by combining \eqref{eq:tphi_invaraince} and \eqref{eq:return_times}, one gets
$$\Phi_k \circ \res{(T^{\times k} \times T_*)}{X^{(k)}} = T_*^{N^{(k)}} \circ \Phi_k.$$ 
Since, by Lemma \ref{prop:distinguish_points}, $\Phi_k$ is a bijection for which $(\Phi_k)_*\res{(\mu^{\otimes k} \otimes \mu^*)}{X^{(k)}} = \mu^*$, we have shown that it yields the desired isomorphism of dynamical systems. 
\end{proof}

\subsection{Relative unique ergodicity}
\label{sect:thm_conf_T_phi}

As before, we assume that $X  =\bbR^+$ and $\mu$ is the Lebesgue measure. The main step in proving Theorem \ref{thm:conf_T_phi} is the following relative unique ergodicity result:
\begin{thm} \label{thm:Poisson_furstenberg}
Let $\bfX := (X, \mu, T)$ be a dynamical system of infinite ergodic index. If the compact extension $(X \times G, \mu \otimes m_G, T_\phi)$ is also of infinite ergodic index, then the  only $(T_\phi)_*$-invariant measure $\rho \in \P((X \times G)^*)$ such that $(\pi_*)_*\rho = \mu^*$ is $\rho = (\mu \otimes m_G)^*$.
\end{thm}

As done in Section \ref{sect:perm_group}, we represent the Poisson extension $((X \times G)^*, (\mu \otimes m_G)^*, (T_\phi)_*) \overset{\pi_*}{\to} (X^*, \mu^*, T_*)$ through a Rokhlin cocycle. We do this using the representation of $((X \times G)^*, (\mu \otimes m_G)^*)$ as a marked point process given in Proposition \ref{prop:marked_point_process}. 

We start by introducing the skew product group $G^\bbN \rtimes \mathfrak{S}(\bbN)$ whose law we define by
$$((h_n)_{n \geq 1}, \tau) \cdot ((g_n)_{n \geq 1}, \s) = ((h_{\s(n)} \cdot g_n)_{n \geq 1}, \tau \circ \s).$$
This group acts on $(G^\bbN, {m_G}^{\otimes \bbN})$ via the maps
$$\chi_{(g_n)_{n \geq 1}, \s}\left((h_n)_{n \geq 1}\right) := \left(g_{\s^{-1}(n)} \cdot h_{\s^{-1}(n)}\right)_{n \geq 1}.$$
Then we define the cocycle from $X^*$ to $G^\bbN \rtimes \mathfrak{S}(\bbN)$ by
$$\overline{\phi}: \o \mapsto (\phi(t_n(\o))_{n \geq 1}, \Psi(\o)).$$
This cocycle induces the following transformation 
$$\begin{array}{cccc}
(T_*)_{\overline{\phi}}: & X^* \times G^\bbN & \to & X^* \times G^\bbN\\
& (\o, (g_n)_{n \geq 1}) & \longmapsto & (T_*\o, \chi_{\overline{\phi}(\o)}((g_n)_{n \geq 1}))
\end{array}.$$
Then, for any $(T_\phi)_*$-invariant measure $\rho$ such that $(\pi_*)_*\rho = \mu^*$, an adaptation of the computation \eqref{eq:equivariance} shows that the map $\Phi$ introduced in \eqref{eq:iso_marked_point_process} gives an isomorphism between the extensions 
$$((X \times G)^*, \rho, (T_\phi)_*) \overset{\pi_*}{\to} (X^*, \mu^*, T_*),$$
and
$$(X^* \times G^\bbN, \Phi_*\rho, (T_*)_{\overline{\phi}}) \to (X^*, \mu^*, T_*).$$

Therefore, to prove Theorem \ref{thm:Poisson_furstenberg}, we need to take a $(T_*)_{\overline{\phi}}$-invariant measure $\l$ such that $\l(\cdot \times G^\bbN) = \mu^*$ and show that $\l = \mu^* \otimes {m_G}^{\otimes \bbN}$. This is what we do bellow:
\begin{proof}[Proof of Theorem \ref{thm:Poisson_furstenberg}]
Let $\l$ be a $(T_*)_{\overline{\phi}}$-invariant measure on $X^* \times G^\bbN$ such that $\l(\cdot \times G^\bbN) = \mu^*$. Fix $k \geq 1$ and set $\l_k$ as the image of $\l$ via $p_k$, the projection on $X^* \times G^k$. The main idea of this proof is to use Proposition \ref{thm:distinguish_points} to distinguish the points $t_1(\o), ..., t_k(\o)$ since they determine the action of $(T_*)_{\overline{\phi}}$ on $g_1, ..., g_k$ and to then view $(t_1(\o), g_1), ..., (t_k(\o), g_k)$ as a compact extension of $t_1(\o), ..., t_k(\o)$ to which Furstenberg's relative unique ergodicity Lemma applies. 

We start our argument by understanding better the dynamics on $g_1, ..., g_k$. Since $(X, \mu, T)$ has infinite ergodic index, Proposition \ref{thm:distinguish_points} tells us that the random time
$$\tilde N^{(k)} (\o, (g_n)_{n \geq 1}) := N^{(k)}(\o)$$
is $\l$-almost-surely finite. Now note that, by definition of $N^{(k)}$, we have
\begin{align*}
p_k \circ (T_*)_{\overline{\phi}}^{\tilde N^{(k)}} (\o, &(g_n)_{n \geq 1}) \\
&= (T_*^{N^{(k)}(\o)}\o, \phi^{(N^{(k)}(\o))}(t_1(\o)) \cdot g_1, ..., \phi^{(N^{(k)}(\o))}(t_k(\o)) \cdot g_k)\\
&= (T_*^{N^{(k)}(\o)}\o, \phi_k(\o) \cdot (g_1, ..., g_k)),
\end{align*}
where we define the cocycle $\phi_k: X^* \arr G^k$ by:
$$\phi_k(\o) := \phi^{(N^{(k)}(\o))}(t_1(\o)), ..., \phi^{(N^{(k)}(\o))}(t_k(\o)),$$
with 
$$\phi^{(p)}(x) := \phi(T^{p-1}x) \cdots \phi(x).$$
Therefore $\l_k$ is invariant under the transformation 
$$\begin{array}{cccc}
(T_*^{N^{(k)}})_{\phi_k}: & X^* \times G^k & \to & X^* \times G^k\\
& (\o, (g_1, ..., g_k)) & \longmapsto & ((T_*)^{N^{(k)}(\o)}\o, \phi_k(\o) \cdot (g_1, ..., g_k))
\end{array}.$$
This map yields a compact extension of $(T_*)^{N^{(k)}}$, but to apply Furstenberg's Lemma (i.e. Lemma \ref{lem:furstenberg}), we still have to prove that 
\begin{eq} \label{eq:final_system}
(X^* \times G^k, \mu^* \otimes m_G^{\otimes k}, (T_*^{N^{(k)}})_{\phi_k})
\end{eq}
is ergodic. 

We recall that $M^{(k)}$ is defined as the return time on $X^{(k)}$ and that $M^{(k)} = N^{(k)} \circ \Phi_k$. Then, Proposition \ref{thm:distinguish_points} tells us that $(X^* \times G^k, \mu^* \otimes m_G^{\otimes k}, (T_*^{N^{(k)}})_{\phi_k})$ is isomorphic to 
\begin{eq} \label{eq:points_distinguished_with_group}
\left(X^{(k)} \times G^k, \res{\left(\mu^{\otimes k} \otimes \mu^*\right)}{X^{(k)}} \otimes m_G^{\otimes k}, \left(\res{(T^{\times k} \times T_*)}{X^{(k)}}\right)_{\widehat{\phi_k}}\right),
\end{eq}
where $\widehat{\phi_k}$ is the cocycle defined by 
$$\widehat{\phi_k} := \phi_k \circ \Phi_k = (\phi^{(M^{(k)})}(x_1), ..., \phi^{(M^{(k)})}(x_k)).$$
However, it is straightforward to check that, up to a permutation of coordinates, \eqref{eq:points_distinguished_with_group} is an induced system of
$$(X^k \times G^k \times X^*, \mu^{\otimes k} \otimes m_G^{\otimes k} \otimes \mu^*, T_\phi^{\times k} \times T_*),$$
which can be written as 
\begin{eq} \label{eq:final_system_infty}
(X \times G, \mu \otimes m_G, T_\phi)^{\otimes k} \otimes (X^*, \mu^*, T_*).
\end{eq}
However, this is a factor of
\begin{eq} \label{eq:final_system_infty_big}
(X \times G, \mu \otimes m_G, T_\phi)^{\otimes k} \otimes ((X \times G)^*, (\mu \otimes m_G)^*, (T_\phi)_*).
\end{eq}
But, since $(X \times G, \mu \otimes m_G, T_\phi)$ is of infinite ergodic index, Corollary \ref{cor:keane_poisson} applies and tells us that \eqref{eq:final_system_infty_big} is ergodic, and therefore \eqref{eq:final_system_infty} is as well. Since an induced system on an ergodic system is also ergodic, this yields that \eqref{eq:final_system} is ergodic. In conclusion, Furstenberg's Lemma implies that $\l_k = \mu^* \otimes m_G^{\otimes k}$.

This being true for every $k \geq 1$, it follows that $\l = \mu^* \otimes m_G^{\otimes \bbN}$. 
\end{proof}

\subsection{Conclusion of the proof of Theorem \ref{thm:conf_T_phi}}
\label{sect:final_proof_conf_Tphi}

We now finish the proof of Theorem \ref{thm:conf_T_phi} by combining our relative unique ergodicity result (Theorem \ref{thm:Poisson_furstenberg}) from the previous section with Theorem \ref{thm:poisson_splitting}. In our application of the splitting result (Theorem \ref{thm:poisson_splitting}), the fact that marginals $\{\nu_i\}_{i \in \llbracket 1, n \rrbracket}$ are Poisson measures will already be known, and the important part will be the fact that the associated joining $\l$ is the product joining.

\begin{proof}[Proof of Theorem \ref{thm:conf_T_phi}]
Let $\bfX := (X, \mu, T)$ be a dynamical system, and $\phi: X \arr G$ a cocycle such that the compact extension $\bfZ := (X \times G, \mu \otimes m_G, T_\phi)$ has infinite ergodic index. 

Let $\l$ be a $(T_\phi)_* \times (T_\phi)_*$-invariant self-joining of $(\mu \otimes m_G)^*$ such that 
\begin{eq} \label{eq:hyp_confined}
(\pi_* \times \pi_*)_*\l = \mu^* \otimes \mu^*.
\end{eq}
Since, as mentioned in Section \ref{sect:intro_poisson}, $(X^*, \mu^*, T_*)$ is weakly mixing, up to replacing $\l$ with one of its ergodic components, we may assume that the system 
$$((X \times G)^* \times (X \times G)^*, \l, (T_\phi)_* \times (T_\phi)_*),$$
is ergodic. Now set $\rho := \Sigma_*\l$. We then use \eqref{eq:hyp_confined} to compute
\begin{align*}
(\pi_*)_*\rho = (\pi_*)_*\Sigma_*\l = \Sigma_*(\pi_* \times \pi_*)_* \l = \Sigma_*(\mu^* \otimes \mu^*) = (2\mu)^*.
\end{align*}
In other words, \eqref{eq:hyp_confined} means that the projection of $\rho$ on $X^*$ is the sum of two independent Poisson point processes of intensity $\mu$, and the result of this sum is a Poisson point process of intensity $2\mu$. Now, since $T_\phi$ has infinite ergodic index, we can apply Theorem \ref{thm:Poisson_furstenberg} to conclude that $\rho = (2\mu \otimes m_G)^*$.

Using again the fact that $T_\phi$ has infinite ergodic index, we can now deduce from Theorem \ref{thm:poisson_splitting} that $\l$ is the product joining
$$\l = (\mu \otimes G)^* \times (\mu \otimes G)^*.$$
\end{proof}

\section{Compact extensions of infinite ergodic index}
\label{sect:ext_infinte_erg_index}

The construction in Theorem \ref{thm:conf_T_phi} relies on a compact extension
$$(X \times G, \mu \otimes m_G, T_\phi),$$
which is of infinite ergodic index. In this section, we present various situations in which such extensions arise. 

\subsection{Abstract constructions}

One way to obtain an example of a compact extension of infinite ergodic index is to start with a dynamical system $\bfZ$ with infinite ergodic index and find a factor $\bfX$ of that system over which $\bfZ$ is a compact extension. The key element for such an approach to work is the following:
\begin{lem}\label{lem:compact_factor}
Let $(Z, \rho, R)$ be an ergodic dynamical system and take a \emph{compact} sub-group $K \subset \mathrm{Aut}(Z, \rho, R)$. Define
$$\B_K := \{A \in \mathcal{B}(Z) \, | \, \forall S \in K, \, R^{-1}A = A \text{ mod } \rho\}.$$
Then $\B_K$ is $\s$-finite and there exist a dynamical system $(X, \mu, T)$ and a factor map $\pi: Z \to X$ such that $\B_K = \s(\pi)$ and $(Z, \rho, R)$ is isomorphic to a compact extension of $(X, \mu, T)$. 
\end{lem}
A proof of this result for probability preserving systems can be found in \cite[Theorem 3.29]{glasner}. One can pass to the infinite measure case by using induced transformations with arguments similar to \cite[\textsection 4]{conditionnal_entropy}. A result similar to Lemma \ref{lem:compact_factor} in the case of non-singular transformations can also be found in \cite[Proposition 2.5]{quotients_nonsingular}.

\begin{exe} We give two applications of Lemma \ref{lem:compact_factor} that yield compact extensions of infinite ergodic index:
\begin{enumerate} [label = (\roman*)]
\item Start with any system $\bfZ := (Z, \rho, R)$ of infinite ergodic index. In that case, the product system $\bfZ \otimes \bfZ := (Z \times Z, \rho \otimes \rho, R \times R)$ is also of infinite ergodic index. On $\bfZ \otimes \bfZ$, we take the symmetric factor, i.e. the factor of sets invariant under the action $S: (z_1, z_2) \mapsto (z_2, z_1)$. Since it is an involution, the group generated by $S$ is $K := \{\mathrm{Id}, S\}$, which is finite and therefore compact. It is also clear that $S$ commutes with $R \times R$, so $K$ is a sub-group of $\mathrm{Aut}(Z \times Z, \rho \otimes \rho, R \times R)$. 
\item We can also use a construction due to Danilenko \cite[Theorem 0.1 (1)]{danilenko}: it is shown that for any countable group $G$, there exists an infinite measure preserving free $G$-action $\{T_g\}_{g \in G}$ on a space $(Z, \rho)$ such that for every $g \in G$ of infinite order, the system $(Z, \rho, T_g)$ has infinite ergodic index. Take a finite group $K$ and apply that result to the group $G := \bbZ \times K$. The resulting action $\{T_{(n, k)}\}_{n \in \bbZ, k \in K}$ gives us our desired system $(Z, \rho, R)$ by taking $R := T_{(1, e_K)}$. It is clear that the group $\t K := \{T_{(0, k)}\}_{k \in K}$ is a sub-group of $\mathrm{Aut}(Z, \rho, R)$ (because the elements of $\t K$ commute with $R$). We thank the referees of the paper for suggesting this example.
\end{enumerate}
\end{exe}

\subsection{An explicit family of examples}

In \cite{silva}, for any finite abelian group $G$, the authors build an explicit family of finite rank transformation with a structure on the columns that yield a clear action of $G$. In that way, their construction yields a explicit compact extension of a rank one transformation.

We present here another description of those extensions, where the cocycle appears explicitly. The work in \cite{silva} deals with extensions of any rank one transformation. However, for simplicity,
 we choose a specific rank one transformation: the infinite Chacon transformation. It is known that it has an infinite ergodic index (see \cite[Theorem 2.2]{def_chacon_infty}). 

\subsubsection{Description of the infinite Chacon transformation}
\label{sect:infty_chacon}

As any rank one transformation, the infinite Chacon transformation can be defined as an increasing union of towers $(\calT_n)_{n \geq 1}$. The tower $\calT_n$ of order $n$, is composed of its levels $\{E_n^{(1)}, ..., E_n^{(h_n)}\}$ such that
$$\calT_n = \bigsqcup_{k=1}^{h_n} E_n^{(k)}.$$
We say that $h_n$ is the height of $\calT_n$. The transformation $T$ acts on $\calT_n$ such that, for $k \in \llbracket 1, h_n -1\rrbracket$, we have
$$T E_n^{(k)} = E_n^{(k+1)}.$$
All levels of $\calT_n$ have same measure under $\mu$, and we denote it by $\mu_n := \mu(E_n^{(k)})$. 

\begin{figure}[h]
\centerline{\includegraphics{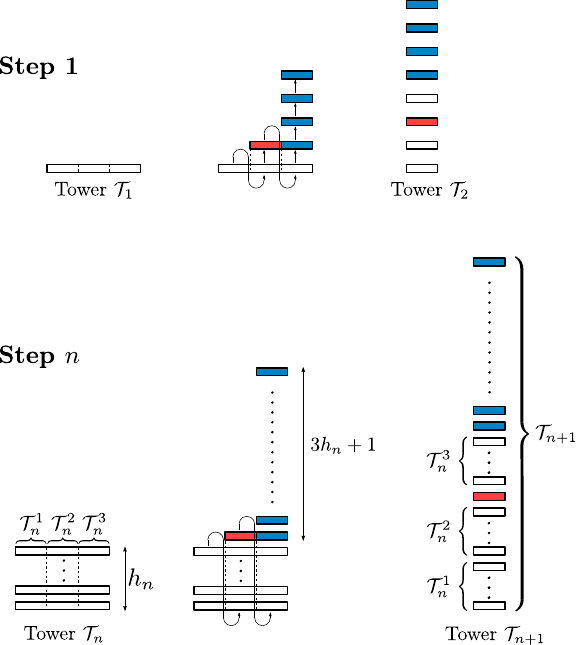}}
\caption{Construction of the infinite Chacon transformation}
\end{figure}

The construction of the sequence $(\calT_n)_{n \geq 1}$ is done inductively. It will be done by taking intervals of $\bbR^+$ with the Lebesgue measure to be the levels of our towers.  Start by taking the interval $[0, 1[$ to be $\calT_1$. Now assume that the tower $\calT_n$ has been built. The construction of $\calT_{n+1}$ goes as follows.

Decompose $\calT_n$ into three disjoint towers of equal measure $\calT_n = \calT_n^1 \sqcup \calT_n^2 \sqcup \calT_n^3$. Specifically, split each level of $\calT_n$ into three intervals of length $\mu_n/3$, then put the left-most interval in $\calT_n^1$, the middle one in $\calT_n^2$ and the right one in $\calT_n^3$. We will call \emph{spacers} a collection of $3h_n+2$ intervals of length $\mu_n/3$, disjoint from $\calT_n$. We put a spacer on top of $\calT_n^2$ and $3h_n +1$ spacers on top of $\calT_n^3$. Once the spacers are in place, we stack $\calT_n^1$, $\calT_n^2$ and $\calT_n^3$ on top of each other, which yields $\calT_{n+1}$. Therefore $\calT_{n+1}$ is a tower of height $2(3h_n+1)$ whose levels each have measure $\mu_n/3$, so $\mu(\calT_{n+1}) \geq 2 \mu(\calT_n)$. Finally, for $k \in \llbracket 1, h_{n+1}-1 \rrbracket$, define $T$ on $E_{n+1}^{(k)}$ as the translation that sends $E_{n+1}^{(k)}$ to $E_{n+1}^{(k+1)}$ (which is possible because they are both intervals of the same length). The transformation is not yet defined on $E_n^{(h_{n+1})}$, that will be done in the next step of the construction.

We end the construction of $(X, \mu, T)$ by setting $X := \bigcup_{n \geq 1} \calT_n$. Since $\mu(\calT_{n+1}) \geq 2 \mu(\calT_n)$, we have $\mu(X) = \infty$.

\subsubsection{Construction of the extension}
\label{sect:extension_infty_erg_index}

The extensions are obtained as follows: let $G$ be a finite abelian group (for which we use additive notation) and $\phi: X \arr G$ be a cocycle constant on the levels of each tower $\calT_n$. To get such a cocycle, one needs to know the value it takes on $\calT_1$, and at each step of the construction of the system $(X, \mu, T)$, when building $\calT_{n+1}$ from $\calT_n$, it suffices to specify the value that the cocycle $\phi$ takes on each spacer that is added. Let $g_{n, 1}$ be the value of $\phi$ on the spacer on top of $\calT_n^2$ and $g_{n, 2}$ be the sum of the values taken by $\phi$ on all of the spacers on top of $\calT_n^3$. Set $f_n$ be the sum of the values taken by $\phi$ on all of the levels of $\calT_n$. We then have $f_{n+1} = 3 f_n + g_{n, 1} + g_{n, 2}$. 

For such a cocycle, the associated compact extension $(X \times G, \mu \otimes m_G, T_\phi)$ is a rank $\#G$ transformation isomorphic to a transformation from \cite{silva}, chosen with the right parameters. Following \cite[Theorem 2.2]{silva}, we know that $(X \times G, \mu \otimes m_G, T_\phi)$ has infinite ergodic index if the following are satisfied:
\begin{enumerate} [label = (\roman*)]
\item The set $\bigcup_{n \geq 1} \{0, f_n, 2 f_n + g_{n, 1}\} \subset G$ generates $G$,
\item For every $n \geq 1$, the couple $(1, 0_G) \in \bbZ \times G$ is in the integer span of 
\begin{align*}
&\{(h_n, f_n), (h_n + 1, f_n + g_{n, 1})\} \cup \\
& \bigcup_{M \geq 1} \{(3h_M + 1, g_{M, 2}), (3h_M + 2, g_{M+1, 1} + g_{M, 2})\},
\end{align*}
in $\bbZ \times G$. 
\end{enumerate}
The condition (i) guaranties the ergodicity of the system, and the second condition gives the infinite ergodic index.

For example, if we take an integer $N$ and denote $A_N$ a spacer used in the construction of $\calT_{N+1}$, then the cocycle $\phi := \mathbbm{1}_{A_N}$ satisfies conditions (i) and (ii), and therefore the resulting compact extension has infinite ergodic index.

\begin{rmq}
We have given several examples of a compact extension of an infinite measure preserving system that is of infinite ergodic index. But we wonder if more general results could be found. Regarding the ergodic index, a significant difference between finite and infinite ergodic theory is the fact that for probability preserving systems, an ergodic index greater than or equal to $2$ is automatically infinite. This is not true in the infinite measure case, but we could consider an intermediate situation: take an extension $\bfZ \overset{\pi}{\to} \bfX$ of $\s$-finite infinite measure systems and assume that $\bfX$ has an infinite ergodic index. Is it possible that $\bfZ$ have a finite ergodic index greater than or equal to $2$ ? Within the family of extensions given by \cite{silva}, the answer is negative. Indeed, for an ergodic system with finite ergodic index of that form, it is shown that condition (i) must hold, so the condition (ii) has to fail (otherwise the ergodic index would be infinite). But \cite[Lemma 4.3]{silva} states that in that case, the system is not totally ergodic, so the 2-fold direct product cannot be ergodic: the ergodic index of the system is $1$.
\end{rmq}

\begin{description} [leftmargin=*] \item[Acknowledgments.]
The authors thank the anonymous referees for their suggestions, in particular regarding the examples given in Section \ref{sect:ext_infinte_erg_index}. 
\end{description}

\bibliographystyle{plain}
\bibliography{Biblio_poisson_extension}

\begin{thebibliography}{10}

\bibitem{aaronson}
Jon Aaronson.
\newblock {\em An introduction to infinite ergodic theory}, volume~50 of {\em
  Math. Surv. Monogr.}
\newblock Providence, RI: American Mathematical Society, 1997.

\bibitem{def_chacon_infty}
Terrence Adams, Nathaniel Friedman, and Cesar~E. Silva.
\newblock Rank-one weak mixing for nonsingular transformations.
\newblock {\em Isr. J. Math.}, 102:269--281, 1997.

\bibitem{benzoniConfined}
S{\'e}verin Benzoni.
\newblock Confined extensions and non-standard dynamical filtrations.
\newblock {\em Stud. Math.}, 276(3):233--270, 2024.

\bibitem{Daley_Vere-Jones}
D.~J. Daley and D.~Vere-Jones.
\newblock {\em An introduction to the theory of point processes. {Vol}. {I}:
  {Elementary} theory and methods.}
\newblock Probab. Appl. New York, NY: Springer, 2nd ed. edition, 2003.

\bibitem{danilenko}
Alexandre~I. Danilenko.
\newblock Funny rank-one weak mixing for nonsingular {Abelian} actions.
\newblock {\em Isr. J. Math.}, 121:29--54, 2001.

\bibitem{conditionnal_entropy}
Alexandre~I. Danilenko and Daniel~J. Rudolph.
\newblock Conditional entropy theory in infinite measure and a question of
  {Krengel}.
\newblock {\em Isr. J. Math.}, 172:93--117, 2009.

\bibitem{silva}
Chris Dodd, Phakawa Jeasakul, Anne Jirapattanakul, Daniel~M. Kane, Becky
  Robinson, Noah~D. Stein, and Cesar~E. Silva.
\newblock Ergodic properties of a class of discrete {Abelian} group extensions
  of rank-one transformations.
\newblock {\em Colloq. Math.}, 119(1):1--22, 2010.

\bibitem{Furstenberg_book}
H.~Furstenberg.
\newblock {\em Recurrence in ergodic theory and combinatorial number theory}.
\newblock M. B. Porter Lect. Princeton University Press, Princeton, NJ, 1981.

\bibitem{glasner}
Eli Glasner.
\newblock {\em Ergodic theory via joinings}, volume 101 of {\em Math. Surv.
  Monogr.}
\newblock Providence, RI: American Mathematical Society (AMS), 2003.

\bibitem{sushi}
{\'E}lise Janvresse, Emmanuel Roy, and Thierry de~la Rue.
\newblock Poisson suspensions and {SuShis}.
\newblock {\em Ann. Sci. {\'E}c. Norm. Sup{\'e}r. (4)}, 50(6):1301--1334, 2017.

\bibitem{ErgPoissonSplittings}
{\'E}lise Janvresse, Emmanuel Roy, and Thierry de~la Rue.
\newblock Ergodic {Poisson} splittings.
\newblock {\em Ann. Probab.}, 48(3):1266--1285, 2020.

\bibitem{marchat}
Francoise~Annie Marchat.
\newblock {\em A {class} {of} {measure}-{preserving} {transformations}
  {arising} {by} {use} {of} {the} {Poisson} {process}}.
\newblock ProQuest LLC, Ann Arbor, MI, 1979.
\newblock Thesis (Ph.D.)--University of California, Berkeley.

\bibitem{Meyerovitch}
Tom Meyerovitch.
\newblock Ergodicity of {Poisson} products and applications.
\newblock {\em Ann. Probab.}, 41(5):3181--3200, 2013.

\bibitem{quotients_nonsingular}
Cesar~E. Silva and Dave Witte.
\newblock On quotients of nonsingular actions whose self-joinings are graphs.
\newblock {\em Int. J. Math.}, 5(2):219--237, 1994.

\end{thebibliography}

\end{document}